\documentclass[11pt]{amsart} \usepackage{latexsym, amssymb, amsthm, times, verbatim, amsmath, stmaryrd}
\usepackage[T1]{fontenc}

\newtheorem{thm}{Theorem}[section]
\newtheorem{lemma}[thm]{Lemma}
\newtheorem{prop}[thm]{Proposition}
\newtheorem{cor}[thm]{Corollary}

\newtheorem{fact}[thm]{Fact}

\theoremstyle{definition}
\newtheorem{df}[thm]{Definition}
\newtheorem{nrmk}[thm]{Remark}

\newtheorem{facts}[thm]{Facts}

\theoremstyle{remark}

\renewcommand{\r}{\mathbb{R}}
\newcommand{\Z}{\mathbb{Z}}

\newcommand{\N}{\mathfrak{N}}
\newcommand{\n}{\mathbb{N}}

\renewcommand{\to}{\rightarrow}

\newcommand\concat{\widehat{\phantom{\eta}}}
\def \balpha{\boldsymbol\alpha}
\def \bbeta{\boldsymbol\beta}

\def \<{\langle}
\def \>{\rangle}

\def \*Z {{{^*}\Z}}
\def \((  {(\!(}
\def \)) {)\!)}

\def \m{\operatorname{M}}

\numberwithin{equation}{section}

\def \u{\mathcal U}

\def \m{\mathcal{M}}
\def \good{\operatorname{good}}

\makeatletter

\def\indsym#1#2{%
  \setbox0=\hbox{$\m@th#1x$}%
  \kern\wd0%
  \hbox to 0pt{\hss$\m@th#1\mid$\hbox to 0pt{$\m@th#1^{#2}$}\hss}%
  \lower.9\ht0\hbox to 0pt{\hss$\m@th#1\smile$\hss}%
  \kern\wd0}

\def\dotminussym#1#2{%
  \setbox0=\hbox{$\m@th#1-$}%
  \kern.5\wd0%
  \hbox to 0pt{\hss\hbox{$\m@th#1-$}\hss}%
  \raise.6\ht0\hbox to 0pt{\hss$\m@th#1.$\hss}%
  \kern.5\wd0}
\newcommand{\dotminus}{\mathbin{\mathpalette\dotminussym{}}}

\def\nindsym#1#2{%
  \setbox0=\hbox{$\m@th#1x$}%
  \kern\wd0%
  \hbox to 0pt{\hss$\m@th#1\not$\kern1.4\wd0\hss}
  \hbox to 0pt{\hss$\m@th#1\mid$\hbox to 0pt{$\m@th#1^{\,#2}$}\hss}%
  \lower.9\ht0\hbox to 0pt{\hss$\m@th#1\smile$\hss}%
  \kern\wd0}

\allowdisplaybreaks[2]

\newcommand{\cstar}{$\mathrm{C}^*$}

\title{Explicit sentences distinguishing McDuff's II$_1$ factors}
\author{Isaac Goldbring, Bradd Hart, and Henry Towsner}
\thanks{Goldbring's work was partially supported by NSF CAREER grant DMS-1349399.}

\address {Department of Mathematics, University of California, Irvine, 340 Rowland Hall (Bldg.\# 400), Irvine, CA, 92697-3875.}
\email{isaac@math.uci.edu}
\urladdr{http://www.math.uci.edu/~isaac}

\address{Department of Mathematics and Statistics, McMaster University, 1280 Main St., Hamilton ON, Canada L8S 4K1}
\email{hartb@mcmaster.ca}
\urladdr{http://ms.mcmaster.ca/~bradd/}
\thanks{Hart's work was partially supported by NSERC}

\address {Department of Mathematics, University of Pennsylvania, 209 South 33rd Street 
Philadelphia, PA 19104-6395.}
\email{htowsner@math.upenn.edu}
\urladdr{http://www.sas.upenn.edu/~htowsner/}
\thanks{Towsner's work was partially supported by NSF grant DMS-1600263}

\begin{document}

\begin{abstract}
Recently, Boutonnet, Chifan, and Ioana proved that McDuff's examples of continuum many pairwise non-isomorphic separable II$_1$ factors are in fact pairwise non-elementarily equivalent.  Their proof proceeded by showing that any ultrapowers of any two distinct McDuff examples are not isomorphic.  In a paper by the first two authors of this paper, Ehrenfeucht-Fra\"isse games were used to find an upper bound on the quantifier complexity of sentences distinguishing the McDuff examples, leaving it as an open question to find concrete sentences distinguishing the McDuff factors.  In this paper, we answer this question by providing such concrete sentences.
\end{abstract}

\maketitle

\section{Introduction}

The first examples of continuum many nonisomorphic separable II$_1$ factors were given by McDuff in \cite{MD2}.  These same examples were shown to be non-elementarily equivalent (in the sense of continuous logic) by Boutonnet, Chifan, and Ioana in \cite{BCI}.  The way they proved that the McDuff factors were not elementarily equivalent was by showing, for any two distinct McDuff examples $\m$ and $\mathcal{N}$ and any two ultrafilters $\u$ and $\mathcal{V}$ on $\n$, that the ultrapowers $\m^\u$ and $\mathcal{N}^{\mathcal{V}}$ were not isomorphic; by standard model-theoretic results, it then follows that $\m$ and $\mathcal{N}$ are not elementarily equivalent.

In \cite{braddisaac}, the techniques in \cite{BCI} were dissected in order to give some information about the sentences distinguishing the McDuff examples.  Indeed, if we enumerate the McDuff examples by $\m_{\balpha}$ for $\balpha\in 2^\omega$ and $k\in \omega$ is the least digit such that $\balpha(k)\not=\bbeta(k)$, then it was shown that there must be a sentence $\theta$ with at most $5k+3$ alternations of quantifiers such that $\theta^{\m_{\balpha}}\not=\theta^{\m_{\bbeta}}$.  The proof there used Ehrenfeucht-Fra\"isse games.  The game-theoretic techniques also hinted at a possible strategy of providing concrete sentences distinguishing the McDuff examples \emph{if} concrete sentences distinguishing examples that differed at the first digit could be obtained.  In \cite[Section 4.1]{braddisaac}, such sentences were obtained, but they lacked the uniformity needed to carry out the strategy outlined there.

In this paper, an even finer analysis of the work in \cite{BCI} is carried out in order to obtain concrete sentences that distinguish McDuff examples that differ at the first digit; this analysis appears in Section 3.  In Section 4, the details of the plan outlined in \cite[Section 4.2]{braddisaac} are given and the inductive construction of sentences distinguishing all of the McDuff examples is elucidated.  We note that the concrete sentences given here that distinguish examples at ``level'' $k$ also have $5k+3$ alternations of quantifiers, agreeing with the game-theoretic bounds predicted in \cite{braddisaac}.

We list here some conventions used throughout the paper.  First, we follow set theoretic notation and view $k\in \omega$ as the set of natural numbers less than $k$:  $k=\{0,1\ldots,k-1\}$.  In particular, $2^k$ denotes the set of functions $\{0,1,\ldots,k-1\}\to \{0,1\}$.  If $\balpha\in 2^k$, then we set $\alpha_i:=\balpha(i)$ for $i=0,1,\ldots,k-1$ and we let $\balpha^\#\in 2^{k-1}$ be such that $\balpha$ is the concatenation of $(\alpha_0)$ and $\balpha^\#$.  If $\balpha\in 2^\omega$, then $\balpha|k$ denotes the restriction of $\balpha$ to $\{0,1,\ldots,k-1\}$.  

Whenever we write a tuple $\vec x$, it will be understood that the length of the tuple is countable (that is, finite or countably infinite).

We will use uppercase letters to denote variables in formulae while their lowercase counterparts will be elements from algebras.  We will use $U$'s and $V$'s (sometimes with subscripts) for variables ranging over the set of unitaries; since unitaries are quantifier-free definable relative to the theory of \cstar-algebras, this convention is harmless.  Of course, we will then use $u$'s and $v$'s for unitaries from specific algebras.


Given a group $\Gamma$ and $a\in \Gamma$, we let $u_a\in L(\Gamma)$ be the canonical unitary associated to $a$.

Fix a von Neumann algebra $\m$.  For $x,y\in \m$, the commutator of $x$ and $y$ is the element $[x,y]:=xy-yx$.   If $A$ is a subalgebra of $\m$, then the relative commutant of $A$ in $\m$ is the set $$A'\cap \m:=\{x\in \m \ | \ [x,a]=0 \text{ for all } a\in A\}.$$In particular, the center of $\m$ is $Z(\m):=\m'\cap \m$.  For a tuple $\vec a$ from $\m$, we write $C(\vec a)$ to denote $A'\cap \m$, where $A$ is the subalgebra of $\m$ generated by the coordinates of $\vec a$.  Technically, this notation should also mention $\m$, but the ambient algebra will always be clear from context, whence we omit any mention of it in the notation.

\section{Preliminaries}

In this section, we gather most of the background material needed in the rest of the paper.  First, we recall McDuff's examples.  Let $\Gamma$ be a countable group.  For $i\geq 1$, let $\Gamma_i$ denote an isomorphic copy of $\Gamma$ and let $\Lambda_i$ denote an isomorphic copy of $\mathbb{Z}$.  Let $\tilde{\Gamma}:=\bigoplus_{i\geq 1}\Gamma_i$.  If $S_\infty$ denotes the group of permutations of $\N$ with finite support, then there is a natural action of $S_\infty$ on $\bigoplus_{i\geq 1} \Gamma$ (given by permutation of indices), whence we may consider the semidirect product $\tilde{\Gamma}\rtimes S_\infty$.  Given these conventions, we can now define two new groups:

$$T_0(\Gamma):=\langle \tilde{\Gamma}, (\Lambda_i)_{i\geq 1} \ | \ [\Gamma_i,\Lambda_j]=0 \text{ for }i\geq j\rangle$$ and

$$T_1(\Gamma):=\langle \tilde{\Gamma}\rtimes S_\infty, (\Lambda_i)_{i\geq 1} \ | \ [\Gamma_i,\Lambda_j]=0 \text{ for }i\geq j\rangle.$$

Note that if $\Delta$ is a subgroup of $\Gamma$ and $\alpha\in \{0,1\}$, then $T_\alpha(\Delta)$ is a subgroup of $T_\alpha(\Gamma)$.  Given a sequence $\balpha\in 2^{\leq \omega}$, we define a group $K_{\balpha}(\Gamma)$ as follows:
\begin{enumerate}
\item $K_{\balpha}(\Gamma):=\Gamma$ if $\balpha=\emptyset$;
\item $K_{\balpha}(\Gamma):=(T_{\alpha_0}\circ T_{\alpha_1}\circ \cdots T_{\alpha_{n-1}})(\Gamma)$ if $\balpha\in 2^n$;
\item $K_{\balpha}$ is the inductive limit of $(K_{\balpha|n})_n$ if $\balpha\in 2^\omega$.
\end{enumerate}

We then set $\m_{\balpha}(\Gamma):=L(T_{\balpha}(\Gamma))$.  When $\Gamma=\mathbb{F}_2$, we simply write $\m_{\balpha}$ instead of $\m_{\balpha}(\mathbb{F}_2)$; these are the McDuff examples referred to the introduction.

Given $n\geq 1$, we let $\tilde{\Gamma}_{\balpha,n}$ denote the subgroup of $T_{\alpha_0}(K_{\balpha^\#}(\Gamma))$ given by the direct sum of the copies of $K_{\balpha^\#}(\Gamma)$ indexed by those $i\geq n$ and we let $P_{\balpha,n}:=L(\tilde{\Gamma}_{\balpha,n})$.  When $\balpha$ has length $1$, we simply refer to $\tilde{\Gamma}_{\emptyset,n}$ as $\tilde{\Gamma}_n$ and $P_{\emptyset,n}$ as $P_n$; if, in addition, $n=1$, then we simply refer to $\tilde{\Gamma}_1$ as $\tilde{\Gamma}$.  As introduced in \cite{braddisaac}, we define a \emph{generalized McDuff ultraproduct corresponding to $\balpha$ and $\Gamma$} to be an ultraproduct of the form $\prod_\u \m_{\balpha}(\Gamma)^{\otimes t_s}$, where $(t_s)$ is a sequence of natural numbers and $\u$ is a nonprincipal ultrafilter on $\n$.

The following definition, implicit in \cite{BCI} and made explicitly in \cite{braddisaac}, is central to our work in this paper.

\begin{df}
We say that a pair of unitaries $u,v $ in a II$_1$ factor $\m$ are {\em good unitaries} if, for all $\zeta \in \m$,
\[ 
\inf_{\eta\in C(u,v)}\|\zeta -\eta\|_2^2 \leq 100 (\|[\zeta,u]\|_2^2 + \|[\zeta,v]\|_2^2). 
\]
\end{df}
In the terminology of \cite{BCI}, this says that $C(u,v)$ is a (2,100)-residual subalgebra of $\m$.

We will need the following key facts, whose proofs are outlined in \cite[Facts 2.6]{braddisaac}.

\begin{facts}\label{key.fact}
Suppose that $\balpha\in 2^{<\omega}$ is nonempty, $\Gamma$ is a countable group, and $(t_s)$ is a sequence of natural numbers.
\begin{enumerate}
\item Suppose that $(m_s)$ and $(n_s)$ are two sequences of natural numbers such that $n_s<m_s$ for all $s$.  Further suppose that $\Gamma$ is an ICC group.  Then $(\prod_\u P_{\balpha,m_s}^{\otimes t_s})'\cap (\prod_\u P_{\balpha,n_s}^{\otimes t_s})$ is a generalized McDuff ultraproduct corresponding to $\balpha^\#$ and $\Gamma$.
\item For any sequence $(n_s)$, there is a pair of good unitaries $\vec u$ from $\prod_\u \m_{\balpha}(\Gamma)^{\otimes t_s}$ such that $\prod_\u P_{\balpha,n_s}^{\otimes t_s}=C(\vec u)$.
\item Given any separable subalgebra $A$ of $\prod_\u \m_{\balpha}(\Gamma)^{\otimes t_s}$, there is a sequence $(n_s)$ such that $\prod_\u P_{\balpha,n_s}^{\otimes t_s}\subset A'\cap \prod_\u \m_{\balpha}(\Gamma)^{\otimes t_s}$.
\end{enumerate}
\end{facts}


\section{Distinguishing examples at level one}

In this section, we will find sentences that distinguish $L(T_0(\Gamma))$ and $L(T_1(\Gamma))$ for nonamenable groups $\Gamma$.  For the purposes of the next section, where the main theorem of the paper is proved, we will actually need to prove a bit more.

In the rest of this paper, we set $\chi(X,U_1,U_2):=100(\|[X,U_1]\|^2_2+\|[X,U_2]\|^2_2)$.

\begin{lemma}\label{keylemma}
Let $\Gamma$ be a countable group and $\alpha\in \{0,1\}$.  For any $t,n\in \n$ with $t\geq 1$, there are $a,b\in \bigoplus_t T_\alpha(\Gamma)$ such that, for any $\zeta\in L(\bigoplus_t T_\alpha(\Gamma))$, we have
$$\|\zeta-\mathbb{E}_{L(\bigoplus_t \widetilde{\Gamma_n})}(\zeta)\|_2^2\leq \chi(\zeta,u_a,u_b)^{L(\bigoplus_t T_\alpha(\Gamma))}.$$
\end{lemma}

\begin{proof}
This follows from \cite[Lemmas 2.6-2.10]{BCI}.
\end{proof}

\begin{df}
We set $\psi_m(V_a,V_b)$ to be the formula
$$\sup_{\vec X,\vec Y}((\inf_U \max_{1\leq i,j\leq m}\|[UX_iU^*,Y_j]\|_2)\dotminus 2\max_{1\leq i\leq m}\sqrt{\chi(X_i,V_a,V_b)})$$ and set $\tau_m:=\inf_{V_a,V_b}\psi_m$. 
\end{df}

\begin{prop}\label{true}
Suppose that $\Gamma$ is a countable group and that $t\geq 1$.  Then for any $m\geq 1$, we have
$$\tau_m^{L(\bigoplus_t T_1(\Gamma))}=0.$$
\end{prop}

\begin{proof}
Apply Lemma \ref{keylemma} with $n=1$, obtaining $a,b\in \bigoplus_t T_1(\Gamma)$.  Let $V_a:=u_a$ and $V_b:=u_b$.  Fix $m$-tuples $\vec x,\vec y\in L(\bigoplus_t T_1(\Gamma))$ and $\epsilon>0$.  For each $i=1,\ldots,m$, we have that
$$\|x_i-\mathbb{E}_{L(\bigoplus_t \widetilde{\Gamma})}(x_i)\|_2^2\leq \chi(x_i,u_a,u_b)^{L(\bigoplus_t T_1(\Gamma))}.$$  In particular, there is $k>0$ such that
$$\|x_i-\mathbb{E}_{L(\bigoplus_t \bigoplus_{j\leq k}{\Gamma_j})}(x_i)\|_2\leq \sqrt{\chi(x_i,u_a,u_b)^{L(\bigoplus_t T_1(\Gamma))}}+\epsilon.$$  
Set $x_i^+:=\mathbb{E}_{L(\bigoplus_t \bigoplus_{j\leq k}{\Gamma_j})}(x_i)$ and $x_i^-:=x_i-x_i^+$.  Let $H_p$ be the subgroup of $T_1(\Gamma)$ generated by $\bigoplus_{j\leq p}\Gamma_i \rtimes S_p$ and $\Lambda_1,\ldots,\Lambda_p$.  For $p>0$ sufficiently large, setting $y_i^+:=\mathbb{E}_{L(\bigoplus_t H_p)}(y_i)$ and $y_i^-:=y_i-y_i^+$, we have $\|y_i^-\|_2\leq \epsilon$.

Choose $\sigma\in S_\infty$ with $\sigma(j)>p$ for all $j\leq m$.  Let $\sigma_1:=(\sigma,\sigma,\ldots,\sigma)\in \bigoplus_t L(T_1(\Gamma))$.  Note that $\sigma_1(\bigoplus_t \bigoplus_{j\leq m}\Gamma_j)\sigma_1^{-1}$ commutes with $L(\bigoplus_t H_p)$.  Let $u\in U(L(\bigoplus_t T_1(\Gamma)))$ be the unitary corresponding to $\sigma_1$.  It follows, for $1\leq i,j\leq m$, that $[ux_i^+u^*,y_j+]=0$, so
$$\|ux_iu^*,y_j\|_2\leq \|[ux_i^+u^*,y_j^-]\|_2+\|[ux_i^-u^*,y_j^+]\|_2+\|[ux_i^-u^*,y_j^-]\|_2.$$

Now $$\|[ux_i^+u^*,y_j^-]\|_2\leq \|ux_i^+u^*y_j^-\|_2+\|y_j^-ux_i^+u^*\|_2\leq 2\|y_j^-\|_2\leq 2\epsilon.$$  Here we use that conditional expectation is a contractive map, so $\|x_i^+\|\leq \|x_i\|\leq 1$.  Since $\|x_i^-\|\leq 2$, one shows that $\|[ux_i^-u^*,y_j^-]\|_2\leq 4\epsilon$ in a similar fashion.  Finally, we have
$$\|[ux_i^-u^*,y_j^+]\|_2\leq 2\|x_i^-\|_2\leq 2(\sqrt{\chi(x_i,u_a,u_b)^{L(\bigoplus_t T_1(\Gamma))}}+\epsilon). $$
Letting $\epsilon$ go to $0$, we get the desired result.
\end{proof}

The following is probably obvious and/or well-known, but in any event:

\begin{lemma}
There is a function $\upsilon:\r^*\to \r^*$ such that, for every $\epsilon>0$ and an inclusion $N\subseteq M$ of II$_1$ factors, if $x\in N$ is such that $d(x,U(M))<\upsilon(\epsilon)$, then $d(x,U(N))<\epsilon$.
\end{lemma}

\begin{proof}
Let $\psi(x):=\max(d(x^*x,1),d(xx^*,1))$.  Then $\psi$ is weakly stable, so there is $\eta>0$ such that if $N$ is any II$_1$ factor and $\psi(x)^N<\eta$, then $d(x,U(N))<\epsilon$.  Let $\upsilon(\epsilon):=\Delta_\psi(\eta)$, where $\Delta_\psi$ is the modulus of uniform continuity for the formula $\psi$.  Now suppose that $N\subseteq M$ are II$_1$ factors and $x\in N$ is such that $d(x,U(M))<\upsilon(\epsilon)$.  Then $\psi(x)^N=\psi(x)^M<\eta$, whence $d(x,U(N))<\epsilon$.
\end{proof}

The following result, which is Lemma 4.6 in \cite{BCI}, will be very important to us.  In what follows, $\pi_n:\Gamma\to \tilde{\Gamma}$ is the canonical embedding with $\pi_n(\Gamma)=\Gamma_n$.

\begin{fact}\label{bci4.6}
Suppose that $\Gamma$ is a countable non-amenable group and $Q$ is a tracial von Neumann algebra.  Then there are $g_1,\ldots,g_m\in \Gamma$ and a constant $C>0$ such that, for any $n\geq 1$, unitaries $v_1,\ldots,v_m\in U(L(\tilde{\Gamma}_{n+1}\otimes Q))$, and $\zeta \in L(T_0(\Gamma))\otimes Q$, we have that
$$\|\zeta\|_2\leq C\sum_{k=1}^m \|u_{\pi_n}(g_k)\zeta-\zeta v_k\|_2.$$
\end{fact}

Note that in the version of \cite{BCI} currently available, the lemma only allows for unitaries in $L(\tilde{\Gamma}_{n+1})$ rather than $L(\tilde{\Gamma}_{n+1}\otimes Q)$.  However, the proof readily adapts to this more general situation and, indeed, the lemma is used in this more general form in the proof of \cite[Lemma 4.4]{BCI}.

 For a nonamenable group $\Gamma$, let $C(\Gamma)$ and $m(\Gamma)$ be as in Fact \ref{bci4.6}.

\begin{prop}\label{false}
Suppose that $\Gamma$ is a nonamenable group.    Let $m=m(\Gamma)$, $C=C(\Gamma)$, and $\delta:=\sqrt{\frac{1}{200(30)^2}\upsilon(\frac{1}{2Cm})}$.  Then whenever $M$ is an intermediate subalgebra $L(T_0(\Gamma))\subseteq M\subseteq L(T_0(\Gamma))\otimes Q$, it follows that $\tau_m^M\geq \delta.$
\end{prop}

\begin{proof}
Suppose, towards a contradiction, that $v_a,v_b\in U(M)$ are such that $\psi_m(v_a,v_b)^M<\delta$.  For each $n$, let $\rho_n:\Gamma\to U(P_n)$ be given by $\rho_n(g):=u_{\pi_n(g)}$.  Since $\bigcup_n (P_n'\otimes Q)$ is dense in $L(T_0(\Gamma))\otimes Q$, there is $n$ sufficiently large so that $$\max(\|[\rho_n(g),v_a]\|_2,\|[\rho_n(g),v_b]\|_2)<\delta$$ for all $g\in \Gamma$.  Fix such an $n$ and set $\rho:=\rho_n$.  It follows that $\chi(\rho(g),v_a,v_b)^M\leq 200\delta^2$ for all $g\in \Gamma$.

By Lemma \ref{keylemma}, we may find $a',b'\in T_0(\Gamma)$ such that, for all $\zeta\in L(T_0(\Gamma))$, we have
$$\|\zeta-\mathbb{E}_{L(\widetilde{\Gamma_{n+1}})}(\zeta)\|_2^2\leq \chi(\zeta,u_{a'},u_{b'})^{L(T_0)(\Gamma)}.$$  For simplicity, write $\mathbb{E}$ instead of $\mathbb{E}_{L(\widetilde{\Gamma_{n+1}})\otimes Q}$.  It then follows that, for all $\zeta \in L(T_0(\Gamma))\otimes Q$, we have
$$\|\zeta-\mathbb{E}(\zeta)\|_2^2\leq \chi(\zeta,u_{a'},u_{b'})^{L(T_0(\Gamma))\otimes Q}.$$  Let $g_1,\ldots,g_m\in \Gamma$ be as in Fact \ref{bci4.6}.  Since $\psi_m(v_a,v_b)^M<\delta$, we may find $u\in U(M)$ such that, for all $1\leq k\leq m$, we have
$$\max(\|[u\rho(g_k)u^*,u_{a'}]\|_2,\|[u\rho(g_k)u^*,u_{b'}]\|_2)<20\sqrt{2}\delta+\delta\leq 30\delta.$$
Let $v_k:=u\rho(g_k)u^*\in U(L(T_0(\Gamma))\otimes Q)$ and let $v_k':=\mathbb{E}(v_k)$.  It follows that $\|v_k-v_k'\|_2^2\leq \chi(v_k,u_{a'},u_{b'})^{L(T_0(\Gamma))\otimes Q}\leq 200(30\delta)^2.$  By the choice of $\delta$, there is $v_k''\in U(L(\widetilde{\Gamma_{n+1}})\otimes Q)$ such that $\|v_k'-v_k''\|_2<\frac{1}{2Cm}$.  By Fact \ref{bci4.6}, we have that
$$\|u\|_2\leq C\sum_k \|\rho(g_k)-uv_k''\|_2\leq C\sum_k \|uv_k-uv_k''\|_2<\frac{1}{2},$$ yielding the desired contradiction.
\end{proof}

\section{The inductive construction}

In this section, we describe an inductive construction of sentences that allows us to carry out the argument hinted at in \cite[Section 4.2]{braddisaac}.  By \cite[Section 4.2]{braddisaac}, we know that centralizers of good unitaries and relative commutants between centralizers of good unitaries are definable sets, whence we can quantify over them.  We actually need to know that we can do this in a uniform manner that does not depend on the ambient II$_1$ factor nor the good unitaries at hand.  Such uniformity is the content of the next lemma.  Note that if $\m$ is a II$_1$ factor, $u_1,u_2\in \m$ are good unitaries and $x\in \m$, then:
\begin{itemize}
\item $d(x,C(u_1,u_2))\leq \sqrt{\chi(x,u_1,u_2)^{\m}}$
\item if $x\in C(u_1,u_2)$, then $\chi(x,u_1,u_2)^{\m}=0$.
\end{itemize}

\begin{lemma}[Quantification Lemma]\label{quant}

\

\begin{enumerate}
\item For every formula $\psi(X,\vec Y,\vec U)$, there are formulae $\hat{\psi}_s(\vec Y,\vec U)$ and $\hat{\psi}_i(\vec Y,\vec U)$such that, for any II$_1$ factor $\m$, any pair of good unitaries $\vec u\in \m$, and any tuple $\vec y\in \m$, we have
$$\hat{\psi}_s(\vec y,\vec u)^{\m}=\sup\{\psi(x,\vec y,\vec u)^{\m} \ : \ x\in C(\vec u)\}$$ and 
$$\hat{\psi}_i(\vec y,\vec u)^{\m}=\inf\{\psi(x,\vec y,\vec u)^{\m} \ : \ x\in C(\vec u)\}.$$
\item For every formula $\rho(X,\vec Y, \vec U_1,\vec U_2)$, there are formulae $\overline{\rho}_s(\vec Y,\vec U_1,\vec U_2)$ and $\overline{\rho}_i(\vec Y,\vec U_1,\vec U_2)$ such that, for any II$_1$ factor $\m$ and any two pairs of good unitaries $\vec u_1,\vec u_2\in \m$ with $C(\vec u_2)\subseteq C(\vec u_1)$ and any tuple $\vec y\in \m$, we have 
$$\overline{\rho}_s(\vec y,\vec u_1,\vec u_2)^{\m}=\sup\{\rho(x,\vec y,\vec u_1,\vec u_2)^{\m} \ : \ x\in C(\vec u_2)'\cap C(\vec u_1))\}$$ and
$$\overline{\rho}_i(\vec y,\vec u_1,\vec u_2)^{\m}=\inf\{\rho(x,\vec y,\vec u_1,\vec u_2)^{\m} \ : \ x\in C(\vec u_2)'\cap C(\vec u_1))\}.$$
\end{enumerate}
\end{lemma}

\begin{proof}
We only prove the infimum statements.  We first prove (1).  Let $\alpha$ be a continuous, nondecreasing function such that $\alpha(0)=0$ and $$|\psi(x,\vec y,\vec u)-\psi(x',\vec y, \vec u)|\leq \alpha(d(x,x'))$$ for all $x,x',\vec y,\vec u$.  We claim that $$\hat{\psi}_i(\vec Y,U_1,U_2):=\inf_X(\psi(X,\vec Y,U_1,U_2)+\alpha(\sqrt{\chi(X,U_1,U_2)}))$$ works.  Fix a II$_1$ factor $\m$, a pair of good unitaries $u_1,u_2\in \m$, and a tuple $\vec y\in \m$.  It is clear that $$\hat{\psi}_i(\vec y,u_1,u_2)^{\m}\leq \inf\{\psi(x,\vec y,u_1,u_2)^{\m}\ : \ x\in C(u_1,u_2)\}.$$  To see the other direction, fix $x,x'\in \m$ and note that
$$\psi(x,\vec y,u_1,u_2)^{\m}\leq \psi(x',\vec y,u_1,u_2)^{\m}+\alpha(d(x,x')),$$ whence, taking the infimum over $x\in C(u_1,u_2)$, we have
$$\inf\{\psi(x,\vec y,u_1,u_2)^{\m}\ : \ x\in C(u_1,u_2)\}\leq \psi(x',\vec y,u_1,u_2)^{\m}+\alpha(\sqrt{\chi(x',u_1,u_2)}^{\m}),$$ whence the desired result follows from taking the infimum over $x'$.

The proof of part (2) proceeds in the same way once we find a formula $\zeta(X,\vec U_1,\vec U_2)$ such that, for any II$_1$ factor $\m$, any two pairs of good unitaries $\vec u_1,\vec u_2\in \m$ such that $C(u_2)\subseteq C(u_1)$, and any $x\in\m$, we have that $d(x,C(u_2)'\cap C(u_1))\leq \zeta(x,\vec u_1,\vec u_2)^{\m}$.  Let $$\mathbb{E}:\m\to C(\vec u_2)'\cap C(\vec u_1), \ \mathbb{E}_1:M\to C(\vec u_1), \text{and }\mathbb{E}_2:C(\vec u_1)\to C(\vec u_2)'\cap C(\vec u_1)$$ be the usual conditional expecations, so $\mathbb{E}=\mathbb{E}_2\circ \mathbb{E}_1$ and $d(x,C(\vec u_2)'\cap C(\vec u_1))=\|x-\mathbb{E}(x)\|_2$.  Note that
$$\|x-\mathbb{E}(x)\|_2\leq \|x-\mathbb{E}_1(x)\|_2+\|\mathbb{E}_1(x)-\mathbb{E}_2(\mathbb{E}_1(x))\|_2.$$  Now $\|x-\mathbb{E}_1(x)\|_2\leq \sqrt{\chi(x,u_{11},u_{12})^{\m}}$.  As proved in \cite[Section 4.2]{braddisaac},
$$\|\mathbb{E}_1(x)-\mathbb{E}_2(\mathbb{E}_1(x))\|_2\leq \sqrt{\sup_{y\in C(\vec u_2)}\|[y,\mathbb{E}_1(x)]\|_2}.$$  Now notice that 
$$\|[y,\mathbb{E}_1(x)]\|_2\leq \|\mathbb{E}_1(x)y-xy\|_2+\|xy-yx\|_2+\|yx-y\mathbb{E}_1(x)\|_2.$$  Let $\psi(X,Y,\vec U_2):=2\chi(X,\vec U_2)+\|XY-YX\|_2$.  It follows that $$\sup_{y\in C(\vec u_2)}\|[y,\mathbb{E}_1(x)]\|_2\leq \hat{\psi}_s(x,\vec u_2)^{\m}.$$  Letting $$\zeta(X,\vec U_1,\vec U_2):=\sqrt{\chi(X,\vec U_1)}+\sqrt{\hat{\psi}_s(X,\vec U_2)}$$ yields the desired formula.  
\end{proof}

Repeatedly applying the Quantification Lemma yields:

\begin{thm}[Relativization Theorem]
For any sentence $\theta$ in prenex normal form, there is a formula $\tilde{\theta}(\vec U_1,\vec U_2)$ such that, for any II$_1$ factor $\m$ and any two pairs of good unitaries $\vec u_1,\vec u_2\in \m$ with $C(\vec u_2)\subseteq C(\vec u_1)$, we have
$$\tilde{\theta}(\vec u_1,\vec u_2)^{\m}=\theta^{C(\vec u_2)'\cap C(\vec u_1)}.$$  Moreover, $\tilde{\theta}$ is also in prenex normal form and has the same number of alternations of quantifiers as $\theta$.
\end{thm}

We now introduce the formulae 

$$\varphi_{\good}(U_1,U_2):=\sup_X\inf_Y \max(\max_{i=1,2}\|[Y,U_i]\|_2,d(X,Y)\dotminus \sqrt{\chi(X,U_1,U_2))})$$ and

$$\varphi^n_{\leq}(\vec Y;\vec U):=\sup_{X\in C(\vec U)}\max_{i=1,\ldots,n} \|[X,Y_i]\|_2.$$

In the definition of $\varphi_{\leq}$, we are abusing notation and really mean the formula one obtains from Lemma \ref{quant}.  In what follows, we will only need to consider $\varphi^3_{\leq}$ and denote this formula simply by $\varphi_{\leq}$.

Note that:
\begin{itemize}
\item If $\m$ is an $\aleph_1$-saturated II$_1$ factor, then $\varphi_{\good}(u_1,u_2)^{\m}=0$ if and only if $u_1,u_2$ is a pair of good unitaries.
\item If $\m$ is any II$_1$ factor, $\vec u\in \m$ is a pair of good unitaries, and $\vec y\in \m^n$ is arbitrary, then $\varphi_{\leq}(\vec y,\vec u)^{\m}=0$ if and only if $\vec y\leq \vec u$.
\end{itemize}

\begin{df}
Given a sentence $\theta$, we recursively define a sequence of sentences $\theta_n$ as follows:  Set $\theta_1:=\theta$.  Supposing that $\theta_n$ has been defined, we set $\theta_{n+1}$ to be the sentence
$$\inf_{\vec U_1}\max(\varphi_{\good}(\vec U_1),\sup_{A}\inf_{\vec U_2}\max(\varphi_{\good}(\vec U_2),\varphi_{\leq}(A,\vec U_1;\vec U_2),\tilde{\theta_n}(\vec U_1,\vec U_2))).$$
\end{df}

When $\theta=\tau_m$, we write $\theta_{m,n}$ for $(\tau_m)_n$.  Here is the main result of this paper:

\begin{thm}
For each nonamenable group $\Gamma$, there is a sequence $(r_n(\Gamma))$ of positive real numbers such that, for any $n,t\in \n$ with $t\geq 1$ and any $\balpha\in 2^n$, we have:
\[
\begin{array}{lr}
\theta_{m,n}^{L(T_\alpha(\Gamma))^{\otimes t}}=0 \text{ for all }m\geq 1 & \text{ if }\balpha(n-1)=1;\\
\theta_{m(\Gamma),n}^{L(T_\alpha(\Gamma))^{\otimes t}}\geq r_n(\Gamma) &\text{ if }\balpha(n-1)=0.
\end{array}
\]
\end{thm}

\begin{proof}
We prove the theorem by induction on $n$.  When $n=1$, the theorem holds by Propositions \ref{true} and \ref{false}.


Inductively suppose that the theorem is true for $n$.  Fix a non-amenable group $\Gamma$.  First suppose that $\balpha\in 2^{n+1}$ is such that $\balpha(n)=1$.  Fix also $m,t\geq 1$.  Let $\m$ be the ultrapower of $L(T_{\balpha}(\Gamma))^{\otimes t}$; by \L os' theorem, it suffices to show that $\theta_{m,n+1}^{\m}=0$.  Fix a pair of good unitaries $\vec u_1>1$.  Given $a\in \m$, we can find a pair of good unitaries $\vec u_2\in \m$ such that $\vec u_2>\{a,\vec u_1\}$.  We then have that $C(\vec u_2)'\cap C(\vec u_1)$ is a generalized McDuff ultraproduct corresponding to $\balpha^{\#}$ and $\Gamma$, whence, by the inductive hypothesis, we have that $\tilde{\theta}_{m,n}(\vec u_1,\vec u_2)^{\m}=\theta^{C(\vec u_2)'\cap C(\vec u_1)}=0$.  It follows that $\theta_{m,n+1}^{\m}=0$.

Now suppose, towards a contradiction, that there is no constant $r_{n+1}(\Gamma)$.  Then for each $l>1$, there is $\balpha_l\in 2^{n+1}$ and $t_l\in \n$ with $t_l\geq 1$ such that $\theta_{m(\Gamma),n+1}^{L(T_{\balpha_l}(\Gamma))^{\otimes t_l}}<\frac{1}{l}$.  Without loss of generality, each $\balpha_l=\balpha$ for some fixed $\balpha\in 2^{n+1}$.  Let $\m:=\prod_\u L(T_{\balpha}(\Gamma))^{\otimes t_l}$, a generalized McDuff ultraproduct corresponding to $\balpha$ and $\Gamma$.  We then have that $\theta_{m(\Gamma),n+1}^{\m}=0$.  Let $\vec u_1$ be a pair of good unitaries witnessing the infimum.  Take any $a>\vec{u_1}$ and then take a pair of good unitaries $\vec u_2>a$ witnessing the infimum for that $a$.  We then have that $C(\vec u_2)'\cap C(\vec u_1)$ is a McDuff ultraproduct corresponding to $\balpha^{\#}$ and $\Gamma$, whence $\tilde{\theta}_{m(\Gamma),n}(\vec u_1,\vec u_2)^{\m}=\theta_n^{C(\vec u_2)'\cap C(\vec u_1)}\geq r_n(\Gamma)$, contradicting the fact that $\tilde{\theta}_{m(\Gamma),n}(\vec u_1,\vec u_2)^{\m}=0$.
\end{proof}

\begin{nrmk}
Note that each $\tau_m$ is equivalent to a formula in prenex normal form that begins with an $\inf$ and has three alternations of quantifiers.  By the construction, it is easy to check, by induction on $n$, that each $\theta_{m,n}$ is equivalent to a formula in prenex normal form that begins with an $\inf$ and has $5n+3$ alternations of quantifiers.  This agrees with the theoretical bounds given in \cite{braddisaac}.
\end{nrmk}



\begin{cor}
Suppose that $\Gamma$ is any countable group and $\balpha,\bbeta\in 2^\omega$ are such that $\balpha|n-1=\bbeta|n-1$, $\balpha(n)=1$, and $\bbeta(n)=0$.  Write $\bbeta=(\bbeta|n+1) \concat \bbeta^*$. Set $m:=m(T_{\bbeta^*}(\Gamma))$ and $r:=r_{n+1}(T_{{\bbeta}^*}(\Gamma))$.  Then $\theta_{m,n+1}^{\m_{\balpha}(\Gamma)}=0$ and $\theta_{m,n+1}^{\m_{\bbeta}(\Gamma)}\geq r$.
\end{cor}


\begin{nrmk}
As pointed out in \cite{BCI}, the results there also show, for any countable group $\Gamma$ and any distinct $\balpha,\bbeta\in 2^{\omega}$, that $\mathrm{C}^*_r(T_{\balpha}(\Gamma))$ and $\mathrm{C}^*_r(T_{\bbeta}(\Gamma))$ are not elementarily equivalent.  Our results here do indeed yield concrete sentences distinguishing these algebras.   As mentioned in \cite{BCI}, the groups $T_{\balpha}(\Gamma)$ are increasing unions of \emph{Powers groups}, whence, by the proof of \cite[Proposition 7.2.3]{munster}, the unique trace on $\mathrm{C}^*_r(\Gamma)$ is \emph{definable}, and uniformly so over all $\balpha\in 2^\omega$.  Consequently, the $\theta_{m,n}$'s can be construed as formulae in the language of \cstar-algebras with imaginary sorts added and, since the completion of $\mathrm{C}^*_r(T_{\balpha}(\Gamma))$ with respect to the GNS representation induced by the unique trace is $\mathcal{M}_{\balpha}(\Gamma)$, we have that the $\theta_{m,n}$'s distinguish the $\mathrm{C}^*_r(T_{\balpha}(\Gamma))$'s as well.

\cite{BCI} also show that $\mathrm{C}^*_r(T_{\balpha}(\Gamma))\otimes \mathcal{Z}$ and $\mathrm{C}^*_r(T_{\bbeta}(\Gamma))\otimes \mathcal{Z}$ are also not elementarily equivalent for distinct $\balpha,\bbeta\in 2^\omega$, where $\mathcal{Z}$ is the Jiang-Su algebra.  Since the unique trace in a monotracial exact $\mathcal{Z}$-stable algebra is definable (and uniformly so) by \cite[Section 3.5]{munster} and the closure of $\mathrm{C}^*_r(T_{\balpha}(\Gamma))\otimes \mathcal{Z}$ in its GNS representation with respect to the unique trace is also $\mathcal{M}(T_{\balpha}(\Gamma))$, we also have concrete sentences distinguishing the $\mathrm{C}^*_r(T_{\balpha}(\Gamma))\otimes \mathcal{Z}$'s \emph{in the case that $\Gamma$ is exact} (e.g. when $\Gamma=\mathbb{F}_2)$.  It would be interesting to know if the unique trace on $\mathrm{C}^*_r(T_{\balpha}(\Gamma))\otimes \mathcal{Z}$ is definable in general, that is, for an arbitrary countable group $\Gamma$.  
\end{nrmk}

\end{document}